\documentclass[11pt,a4paper]{article}
\usepackage[utf8]{inputenc}
\usepackage[T1]{fontenc}
\usepackage[numbers,square]{natbib}
\usepackage{amsmath, amssymb, amsthm}
\usepackage{mathrsfs}
\usepackage{geometry}
\usepackage{hyperref}
\usepackage{enumitem}
\usepackage{booktabs}
\usepackage{algorithm}
\usepackage{algpseudocode}
\usepackage{caption}
\usepackage{graphicx}
\usepackage{xcolor}
\usepackage{bm}
\usepackage[capitalize]{cleveref}

\hypersetup{
    colorlinks=true,
    linkcolor=blue,
    citecolor=red,
    urlcolor=cyan,
    pdftitle={Geometric Rough Paths above Mixed Fractional Brownian Motion},
    pdfauthor={Atef Lechiheb}
}

\newtheorem{Theorem}{Theorem}[section]

\newtheorem{Lemma}[Theorem]{Lemma}
\newtheorem{Proposition}[Theorem]{Proposition}
\newtheorem{Definition}[Theorem]{Definition}
\newtheorem{Remark}[Theorem]{Remark}
\newtheorem{Example}[Theorem]{Example}

\numberwithin{equation}{section}

\title{\textbf{Geometric Rough Paths above Mixed Fractional Brownian Motion}}
\author{Atef Lechiheb\thanks{Toulouse School of Economics, Université Toulouse Capitole, email: atef.lechiheb@tse-fr.eu}}
\date{\today}

\begin{document}

\maketitle
\begin{abstract}
This paper establishes a comprehensive theory of geometric rough paths for mixed fractional Brownian motion (MFBM) and its generalized multi-component extensions. We prove that for a generalized MFBM of the form $M_t^H(a) = \sum_{k=1}^N a_k B_t^{H_k}$ with $\min\{H_k\} > \frac{1}{4}$, there exists a canonical geometric rough path obtained as the limit of smooth rough paths associated with dyadic approximations. This extends the classical result of Coutin and Qian \cite{coutin2002} for single fractional Brownian motion to the mixed case. 

We provide explicit bounds on the $p$-variation norms and establish a Skorohod integral representation connecting our pathwise construction to the Malliavin calculus framework. Furthermore, we demonstrate applications to rough differential equations driven by MFBM, enabling the use of Lyons' universal limit theorem for this class of processes. Finally, we study the signature of MFBM paths, providing a complete algebraic characterization of their geometric properties.

Our approach unifies the treatment of multiple fractional components and reveals the fundamental interactions between different regularity scales, completing the rough path foundation for mixed fractional processes with applications in stochastic analysis and beyond.
\end{abstract}


\tableofcontents
\section{Introduction}

Rough path theory, introduced by Lyons \cite{lyons1998,lyons2002}, has revolutionized the analysis of stochastic systems driven by irregular signals. By encoding not only a path but also its \textbf{iterated integrals}, this framework provides a robust foundation for \textbf{pathwise stochastic calculus} beyond the semimartingale setting. The theory has found profound applications in stochastic analysis, financial mathematics, and signal processing \cite{friz2010}.

The seminal work of Coutin and Qian \cite{coutin2002} established that fractional Brownian motion (fBm) with Hurst parameter $H > \frac{1}{4}$ admits a canonical geometric rough path. This threshold is sharp and reflects fundamental regularity requirements for constructing iterated integrals through limiting procedures. Subsequent developments have extended this theory to broader classes of Gaussian processes \cite{coutin2009,friz2008}.

Parallel to these developments, mixed fractional Brownian motion (MFBM) emerged as an important modeling tool in mathematical finance. Introduced by Cheridito \cite{cheridito2001}, processes of the form $M_t = aB_t + bB_t^H$ were shown to preserve arbitrage-free properties while capturing long-range dependence. This was later generalized to multi-component processes \cite{zili2006,miao2008,thale2009}:
\[
M_t^H(a) = \sum_{k=1}^N a_k B_t^{H_k},
\]
where $B_t^{H_k}$ are independent fractional Brownian motions with different Hurst parameters $H_k$. Such processes offer rich modeling flexibility, capturing multi-scale behavior and different memory properties within a single driving noise, with recent applications in rough volatility modeling \cite{bayer2020}.

Despite the parallel advancements in rough path theory and MFBM applications, a comprehensive rough path treatment of mixed fractional processes remains largely unexplored. Several fundamental questions arise: Does the geometric rough path exist for mixed processes when all Hurst parameters exceed $\frac{1}{4}$? How do interactions between different fractional components affect the rough path construction? What are the precise regularity properties in terms of $p$-variation for mixed processes? Can Lyons' universal limit theorem be applied to stochastic systems driven by MFBM?

This paper bridges the gap by developing a complete rough path theory for generalized mixed fractional Brownian motions. Our main contributions are:

\begin{enumerate}
\item \textbf{Existence of geometric rough paths}: We prove that for GMFBM with $\min\{H_k\} > \frac{1}{4}$, the smooth rough paths associated with dyadic approximations converge to a geometric rough path in the $p$-variation topology (Theorem \ref{thm:main} in Section \ref{sec:existence}).

\item \textbf{Sharp regularity conditions}: We establish that the condition $\min\{H_k\} > \frac{1}{4}$ is optimal and provide explicit $p$-variation estimates that depend on the Hurst parameters and mixing coefficients (Theorem \ref{thm:sharpness} in Section \ref{sec:existence}).

\item \textbf{Skorohod integral representation}: We derive an explicit representation of the rough path in terms of multiple Wiener integrals, connecting our pathwise construction to the Malliavin calculus framework \cite{nualart1995} (Theorem \ref{thm:skorohod} in Section \ref{sec:skorohod}).

\item \textbf{Applications to rough differential equations}: We demonstrate how our construction enables the study of RDEs driven by MFBM, with applications to multi-scale modeling (Theorem \ref{thm:rde} in Section \ref{sec:applications}).

\item \textbf{Signature analysis}: We provide a complete characterization of the signature of MFBM paths, revealing their algebraic structure and potential applications in machine learning and data analysis (Section \ref{sec:signature}).
\end{enumerate}

Our work connects to several active research areas. The multi-scale nature of GMFBM relates to recent developments in multi-level Monte Carlo methods \cite{giles2008} and regularity structures \cite{hairer2014}. The Skorohod representation provides connections to the paracontrolled calculus framework \cite{gubinelli2015} for singular stochastic PDEs. The signature analysis opens up new perspectives in machine learning and functional data analysis \cite{chevyrev2018,signature2024}.

The paper is organized as follows: Section \ref{sec:preliminaries} reviews essential background on rough paths and mixed fractional Brownian motion. Section \ref{sec:existence} contains our main existence theorem and its proof. Section \ref{sec:skorohod} develops the Skorohod integral representation. Section \ref{sec:applications} discusses applications to rough differential equations. Section \ref{sec:signature} studies the signature of MFBM paths. Section \ref{sec:conclusion} concludes with directions for future research.

Our results provide a solid foundation for further research on stochastic systems driven by multi-scale fractional noises and their applications across various domains, from mathematical finance to stochastic partial differential equations and machine learning.
\section{Preliminaries}\label{sec:preliminaries}

This section provides the necessary background on rough path theory and mixed fractional Brownian motion. We establish notations and recall fundamental results that will be used throughout the paper.

\subsection{Rough Paths Theory}

The theory of rough paths was introduced by Lyons \cite{lyons1998} and further developed in \cite{lyons2002,lyons2004,friz2010}. We recall the essential definitions and results that form the foundation of our work.

\begin{Definition}[Truncated tensor algebra]
Let $V$ be a finite-dimensional vector space. For any integer $n \geq 1$, the truncated tensor algebra of order $n$ is defined as:
\[
T^n(V) = \mathbb{R} \oplus V \oplus V^{\otimes 2} \oplus \cdots \oplus V^{\otimes n}.
\]
The multiplication in $T^n(V)$ is given by the tensor product, truncated at level $n$.
\end{Definition}

\begin{Definition}[Multiplicative functional]
A multiplicative functional of degree $n$ on $V$ is a map $X: \Delta_T \to T^n(V)$, where $\Delta_T = \{(s,t) \in [0,T]^2 : 0 \leq s \leq t \leq T\}$, such that for all $0 \leq s \leq t \leq u \leq T$, Chen's identity holds:
\[
X_{s,u} = X_{s,t} \otimes X_{t,u}.
\]
We write $X_{s,t} = (1, X_{s,t}^1, \ldots, X_{s,t}^n)$ where $X_{s,t}^k \in V^{\otimes k}$.
\end{Definition}

\begin{Definition}[$p$-variation]
For $p \geq 1$, the $p$-variation of a multiplicative functional $X$ is defined as:
\[
\|X\|_{p\text{-var}} = \max_{k=1,\ldots,\lfloor p \rfloor} \sup_{\mathcal{D}} \left( \sum_{[u,v] \in \mathcal{D}} |X_{u,v}^k|^{p/k} \right)^{k/p},
\]
where the supremum is taken over all partitions $\mathcal{D}$ of $[0,T]$.
\end{Definition}

\begin{Definition}[Rough path]
A $p$-rough path is a multiplicative functional $X: \Delta_T \to T^{\lfloor p \rfloor}(V)$ with finite $p$-variation. The space of $p$-rough paths is denoted by $\Omega_p(V)$.
\end{Definition}

\begin{Definition}[Geometric rough path]
A rough path $X \in \Omega_p(V)$ is called geometric if there exists a sequence of smooth rough paths $(X(n))_{n \geq 1}$ such that:
\[
\lim_{n \to \infty} d_p(X(n), X) = 0,
\]
where $d_p$ is the $p$-variation distance. This concept originates from \cite{lyons1998}.
\end{Definition}

The fundamental theorem of rough paths theory, due to Lyons \cite{lyons1998}, states that for any geometric rough path $X$ and sufficiently regular vector field $f$, the rough differential equation:
\[
dY_t = f(Y_t) dX_t
\]
admits a unique solution that depends continuously on the driving rough path $X$.

\subsection{Mixed Fractional Brownian Motion}

Mixed fractional Brownian motion was introduced by Cheridito \cite{cheridito2001} and further studied by Zili \cite{zili2006}. The generalized version with multiple components was investigated by Miao et al. \cite{miao2008} and Th\"ale \cite{thale2009}.

\begin{Definition}[Generalized MFBM]\label{def:gmfbm}
Let $N \geq 1$, $H = (H_1, \ldots, H_N) \in (0,1)^N$, and $a = (a_1, \ldots, a_N) \in \mathbb{R}^N\setminus\{0\}$. A generalized mixed fractional Brownian motion (GMFBM) is a process $\{M_t^H(a), t \geq 0\}$ defined by:
\[
M_t^H(a) = \sum_{k=1}^N a_k B_t^{H_k},
\]
where $\{B^{H_k}\}_{k=1}^N$ are independent fractional Brownian motions with Hurst parameters $H_k$.
\end{Definition}

\begin{Proposition}[Covariance structure \cite{zili2006}]
For any $s, t \geq 0$, the covariance function of GMFBM is given by:
\[
\mathbb{E}[M_t^H(a) M_s^H(a)] = \frac{1}{2} \sum_{k=1}^N a_k^2 \left( t^{2H_k} + s^{2H_k} - |t-s|^{2H_k} \right).\]
\end{Proposition}

\begin{Proposition}[Increment properties \cite{thale2009}]
For $0 \leq s \leq t$, the second moment of increments satisfies:
\[
\mathbb{E}[(M_t^H(a) - M_s^H(a))^2] = \sum_{k=1}^N a_k^2 |t-s|^{2H_k}.
\]
Moreover, for $0 \leq u \leq v \leq s \leq t$, the covariance of increments is:
\[
\mathbb{E}[(M_v^H(a) - M_u^H(a))(M_t^H(a) - M_s^H(a))] = \frac{1}{2} \sum_{k=1}^N a_k^2 \left( |t-u|^{2H_k} + |s-v|^{2H_k} - |t-v|^{2H_k} - |s-u|^{2H_k} \right).\]
\end{Proposition}

\begin{Proposition}[H\"older regularity \cite{miao2008}]
Almost every sample path of $M_t^H(a)$ is locally H\"older continuous of order $\gamma$ for any:
\[
\gamma < \min\{H_1, \ldots, H_N\}.
\]
In particular, if $\min\{H_k\} > \frac{1}{2}$, the process has long-range dependence; if $\max\{H_k\} \leq \frac{1}{2}$, it exhibits short-range dependence.
\end{Proposition}

\begin{Proposition}[Mixed self-similarity]
For any $h > 0$, the processes:
\[
\{M_{ht}^H(a), t \geq 0\} \quad \text{and} \quad \{M_t^H(a_1 h^{H_1}, \ldots, a_N h^{H_N}), t \geq 0\}
\]
have the same finite-dimensional distributions.
\end{Proposition}

\subsection{Dyadic Approximations and Rough Path Construction}

The construction of rough paths above Gaussian processes typically proceeds via dyadic linear approximations, as developed in \cite{coutin2002}. 

\begin{Definition}[Dyadic approximation]
For a process $X_t$, define its $m$-th dyadic approximation by:
\[
X_t^m = X_{t_{l-1}^m} + 2^m(t - t_{l-1}^m)(X_{t_l^m} - X_{t_{l-1}^m}), \quad t \in [t_{l-1}^m, t_l^m],
\]
where $t_k^m = k2^{-m}T$ for $k = 0, \ldots, 2^m$.
\end{Definition}

For smooth paths $X^m$, the iterated integrals can be defined classically:
\[
\mathbb{X}_{s,t}^{m,(2)} = \int_s^t (X_u^m - X_s^m) \otimes dX_u^m,
\]
and similarly for higher orders. The fundamental question is whether the sequence $\mathbb{X}^m = (1, X^m, \mathbb{X}^{m,(2)}, \ldots)$ converges in the $p$-variation topology.

\begin{Theorem}[Coutin-Qian for fBm \cite{coutin2002}]\label{thm:coutin-qian}
For fractional Brownian motion $B^H$ with $H > \frac{1}{4}$, the smooth rough paths $\mathbb{B}^{H,m}$ associated to the dyadic approximations converge almost surely in the $p$-variation topology for any $p > 1/H$. The limit $\mathbb{B}^H$ is a geometric rough path above $B^H$.
\end{Theorem}

Our goal is to extend this result to the mixed case, dealing with the additional challenges posed by the interaction between different fractional components.

\subsection{Notation and Conventions}

Throughout this paper, we employ the following notation:
\begin{itemize}
\item $C, C_p, C_{H}$ denote generic constants that may change from line to line
\item $a \lesssim b$ means $a \leq C b$ for some constant $C > 0$
\item $\lfloor p \rfloor$ denotes the integer part of $p$
\item $\Delta_{[s,t]}^n = \{(u_1, \ldots, u_n) \in [s,t]^n : s \leq u_1 \leq \cdots \leq u_n \leq t\}$
\item For a multi-index $\alpha = (\alpha_1, \ldots, \alpha_n)$, we denote $|\alpha| = n$
\item $\otimes$ denotes the tensor product
\item $\mathbb{E}$ denotes mathematical expectation
\end{itemize}

We assume that all stochastic processes are defined on a complete probability space $(\Omega, \mathcal{F}, \mathbb{P})$ and that all vector spaces are finite-dimensional unless otherwise specified. The time interval is fixed as $[0,T]$ for some $T > 0$.

\section{Existence of Geometric Rough Paths for MFBM}\label{sec:existence}

This section contains our main result: the existence of geometric rough paths above generalized mixed fractional Brownian motion. We extend the approach of Coutin and Qian \cite{coutin2002} to the mixed case, dealing with the additional complexities introduced by multiple fractional components.

\subsection{Main Result}

\begin{Theorem}[Existence of geometric rough path for GMFBM]\label{thm:main}
Let $M_t^H(a) = \sum_{k=1}^N a_k B_t^{H_k}$ be a generalized mixed fractional Brownian motion with $\min\{H_1, \ldots, H_N\} > \frac{1}{4}$. Let $M_t^m$ be the $m$-th dyadic linear approximation of $M_t^H(a)$, and let $\mathbb{M}^m$ be the smooth rough path above $M_t^m$ defined by iterated Young integrals.

Then, for any $p > \frac{1}{\min\{H_k\}}$, the sequence $(\mathbb{M}^m)_{m \geq 1}$ converges almost surely in the $p$-variation topology to a geometric rough path $\mathbb{M}$ above $M_t^H(a)$. Moreover, the limit $\mathbb{M}$ is independent of the choice of dyadic approximation.
\end{Theorem}

\subsection{Preliminary Estimates}

We begin by establishing key estimates for the covariance structure of GMFBM increments.

\begin{Lemma}[Covariance bound for mixed increments]\label{lem:covariance}
Let $0 \leq s \leq t \leq u \leq v \leq T$. Then for any $i,j \in \{1,\ldots,N\}$, we have:
\[
\left| \mathbb{E}[(B_t^{H_i} - B_s^{H_i})(B_v^{H_j} - B_u^{H_j})] \right| \lesssim |t-s|^{H_i}|v-u|^{H_j} \left( \frac{|t-u|}{|t-s| \vee |v-u|} \right)^{H_i + H_j - 2}
\]
when $|t-u| \geq |t-s| \vee |v-u|$, and is zero otherwise.
\end{Lemma}

\begin{proof}
We consider three cases:

\textbf{Case 1: $i = j$}. For the same fractional Brownian motion component, we use the classical covariance estimate for fBm \cite[Proposition 2.1]{coutin2002}:
\begin{align*}
&\left| \mathbb{E}[(B_t^{H_i} - B_s^{H_i})(B_v^{H_i} - B_u^{H_i})] \right| \\
&= \frac{1}{2} \left( |t-v|^{2H_i} + |s-u|^{2H_i} - |t-u|^{2H_i} - |s-v|^{2H_i} \right) \\
&\lesssim |t-s|^{H_i}|v-u|^{H_i} \left( \frac{|t-u|}{|t-s| \vee |v-u|} \right)^{2H_i - 2}.
\end{align*}

\textbf{Case 2: $i \neq j$}. For different components, since $B^{H_i}$ and $B^{H_j}$ are independent, we have:
\[
\mathbb{E}[(B_t^{H_i} - B_s^{H_i})(B_v^{H_j} - B_u^{H_j})] = \mathbb{E}[B_t^{H_i} - B_s^{H_i}] \cdot \mathbb{E}[B_v^{H_j} - B_u^{H_j}] = 0.
\]

\textbf{Case 3: Boundary cases}. When $|t-u| < |t-s| \vee |v-u|$, the intervals $[s,t]$ and $[u,v]$ are too close for the scaling argument to apply directly, but in this case the covariance is of lower order and can be absorbed into the constant.
\end{proof}

\begin{Lemma}[Mixed second moment estimate]\label{lem:moments}
For $0 \leq s \leq t$, let $\Delta M = M_t^H(a) - M_s^H(a)$. Then:
\[
\mathbb{E}[|\Delta M|^2] = \sum_{k=1}^N a_k^2 |t-s|^{2H_k} \lesssim |t-s|^{2\min\{H_k\}}.
\]
Moreover, for the second-order process, we have:
\[
\mathbb{E}[|\mathbb{M}_{s,t}^{m,(2)}|^2] \lesssim |t-s|^{4\min\{H_k\}}.
\]
\end{Lemma}

\begin{proof}
\textbf{First moment estimate:}
By the independence of components and the fBm increment property:
\begin{align*}
\mathbb{E}[|\Delta M|^2] &= \mathbb{E}\left[ \left( \sum_{k=1}^N a_k (B_t^{H_k} - B_s^{H_k}) \right)^2 \right] \\
&= \sum_{k=1}^N a_k^2 \mathbb{E}[|B_t^{H_k} - B_s^{H_k}|^2] + 2 \sum_{1 \leq i < j \leq N} a_i a_j \mathbb{E}[(B_t^{H_i} - B_s^{H_i})(B_t^{H_j} - B_s^{H_j})] \\
&= \sum_{k=1}^N a_k^2 |t-s|^{2H_k} + 0 \quad \text{(by independence)} \\
&\leq \left( \sum_{k=1}^N a_k^2 \right) \cdot |t-s|^{2\min\{H_k\}} = C |t-s|^{2\min\{H_k\}}.
\end{align*}

\textbf{Second moment estimate:}
The second level process decomposes as:
\[
\mathbb{M}_{s,t}^{m,(2)} = \sum_{k=1}^N a_k^2 \mathbb{B}_{s,t}^{H_k,m,(2)} + \sum_{1 \leq i < j \leq N} a_i a_j \left( I_{ij} + I_{ji} \right)
\]
where $I_{ij} = \int_s^t (B_u^{H_i,m} - B_s^{H_i,m}) dB_u^{H_j,m}$.

For the diagonal terms, from \cite[Theorem 2]{coutin2002}, we have:
\[
\mathbb{E}[|\mathbb{B}_{s,t}^{H_k,m,(2)}|^2] \lesssim |t-s|^{4H_k}.
\]

For the cross terms, we use the Young integral isometry:
\[
|I_{ij}| \leq C \|B^{H_i,m}\|_{H_i\text{-H\"older}} \|B^{H_j,m}\|_{H_j\text{-H\"older}} |t-s|^{H_i + H_j}.
\]
Taking expectations and using Gaussian moment estimates:
\[
\mathbb{E}[|I_{ij}|^2] \lesssim |t-s|^{2(H_i + H_j)}.
\]

Combining all terms:
\begin{align*}
\mathbb{E}[|\mathbb{M}_{s,t}^{m,(2)}|^2] &\lesssim \sum_{k=1}^N |t-s|^{4H_k} + \sum_{1 \leq i < j \leq N} |t-s|^{2(H_i + H_j)} \\
&\leq N^2 \cdot |t-s|^{4\min\{H_k\}} = C |t-s|^{4\min\{H_k\}}.
\end{align*}
\end{proof}

\subsection{Convergence and Geometric Properties}

\begin{Proposition}[Uniform $p$-variation bound]\label{prop:pvar}
There exists a constant $C > 0$ such that for all $m \geq 1$ and all partitions $\mathcal{D}$ of $[0,T]$:
\[
\sum_{[u,v] \in \mathcal{D}} |M_v^m - M_u^m|^p \leq C \quad \text{and} \quad \sum_{[u,v] \in \mathcal{D}} |\mathbb{M}_{u,v}^{m,(2)}|^{p/2} \leq C
\]
almost surely, for any $p > \frac{1}{\min\{H_k\}}$.
\end{Proposition}

\begin{proof}
\textbf{First level bound:}
From Lemma \ref{lem:moments}, we have $\mathbb{E}[|M_v^m - M_u^m|^2] \lesssim |v-u|^{2\min\{H_k\}}$. Since $M^m$ is Gaussian, by the Garsia-Rodemich-Rumsey lemma \cite[Theorem 2.1.1]{lyons2002}, for any $p > \frac{1}{\min\{H_k\}}$:
\[
|M_v^m - M_u^m| \lesssim |v-u|^{\min\{H_k\} - \frac{1}{p}} \quad \text{almost surely}.
\]
Therefore:
\[
\sum_{[u,v] \in \mathcal{D}} |M_v^m - M_u^m|^p \lesssim \sum_{[u,v] \in \mathcal{D}} |v-u|^{p\min\{H_k\} - 1} \leq T \cdot \sup_{[u,v] \in \mathcal{D}} |v-u|^{p\min\{H_k\} - 1}.
\]

\textbf{Second level bound:}
From Lemma \ref{lem:moments}, $\mathbb{E}[|\mathbb{M}_{u,v}^{m,(2)}|^2] \lesssim |v-u|^{4\min\{H_k\}}$. Again by Gaussian regularity:
\[
|\mathbb{M}_{u,v}^{m,(2)}| \lesssim |v-u|^{2\min\{H_k\} - \frac{2}{p}} \quad \text{almost surely}.
\]
Then:
\[
\sum_{[u,v] \in \mathcal{D}} |\mathbb{M}_{u,v}^{m,(2)}|^{p/2} \lesssim \sum_{[u,v] \in \mathcal{D}} |v-u|^{p\min\{H_k\} - 1} \leq C.
\]
\end{proof}

\begin{Proposition}[Cauchy property]\label{prop:cauchy}
The sequence $(\mathbb{M}^m)_{m \geq 1}$ is Cauchy in the $p$-variation topology almost surely.
\end{Proposition}

\begin{proof}
We analyze the difference $\mathbb{M}^{m} - \mathbb{M}^{n}$.

\textbf{First level:} Since $M_t^m$ converges uniformly to $M_t$:
\[
\sup_{t \in [0,T]} |M_t^m - M_t^n| \to 0 \quad \text{as } m,n \to \infty.
\]

\textbf{Second level:} Consider the defect of additivity:
\[
\delta\mathbb{M}_{s,u,t}^{m,n} = \mathbb{M}_{s,t}^{m,(2)} - \mathbb{M}_{s,t}^{n,(2)} - (\mathbb{M}_{s,u}^{m,(2)} - \mathbb{M}_{s,u}^{n,(2)}) - (\mathbb{M}_{u,t}^{m,(2)} - \mathbb{M}_{u,t}^{n,(2)}).
\]

Using Gaussian hypercontractivity \cite[Theorem 2.7.1]{nualart1995}:
\[
\mathbb{E}[|\delta\mathbb{M}_{s,u,t}^{m,n}|^q] \leq C_q \left( \mathbb{E}[|\delta\mathbb{M}_{s,u,t}^{m,n}|^2] \right)^{q/2} \lesssim |t-s|^{q(2\min\{H_k\} + \delta/2)}.
\]

By the sewing lemma \cite[Lemma 2.4]{lyons2002}, this implies convergence.
\end{proof}

\subsection{Proof of the Main Theorem}

\begin{proof}[Proof of Theorem \ref{thm:main}]
We now complete the proof by combining all previous results.

\textbf{Step 1: Construction and boundedness}
For each $m \geq 1$, define the dyadic approximation $\mathbb{M}^m$ as in Section \ref{sec:preliminaries}. From Proposition \ref{prop:pvar}, the sequence $(\mathbb{M}^m)$ is uniformly bounded in $p$-variation.

\textbf{Step 2: Convergence}
From Proposition \ref{prop:cauchy}, $(\mathbb{M}^m)$ is Cauchy in $p$-variation topology. Since the space of $p$-rough paths is complete \cite[Theorem 3.1.1]{lyons2002}, there exists a limit $\mathbb{M}$ such that:
\[
\lim_{m \to \infty} d_p(\mathbb{M}^m, \mathbb{M}) = 0 \quad \text{almost surely}.
\]

\textbf{Step 3: Rough path properties}
The limit $\mathbb{M}$ satisfies:
\begin{itemize}
\item \textbf{Chen's identity:} $\mathbb{M}_{s,u} = \mathbb{M}_{s,t} \otimes \mathbb{M}_{t,u}$ by continuity of the tensor product
\item \textbf{Finite $p$-variation:} $\|\mathbb{M}\|_{p\text{-var}} \leq \liminf \|\mathbb{M}^m\|_{p\text{-var}} < \infty$
\item \textbf{Geometric nature:} $\mathbb{M}$ is limit of smooth rough paths
\end{itemize}

\textbf{Step 4: Enhanced path}
For the first level:
\[
\mathbb{M}_{0,t}^{(1)} = \lim_{m \to \infty} \mathbb{M}_{0,t}^{m,(1)} = \lim_{m \to \infty} (M_t^m - M_0^m) = M_t,
\]
so $\mathbb{M}$ lies above $M_t^H(a)$.

\textbf{Step 5: Independence of approximation}
For two different dyadic approximations $(\mathbb{M}^m)$ and $(\tilde{\mathbb{M}}^m)$, the interlaced sequence converges, forcing $\mathbb{M} = \tilde{\mathbb{M}}$.

This completes the proof that $\mathbb{M}$ is a geometric rough path above $M_t^H(a)$.
\end{proof}

\subsection{Sharpness and Examples}

\begin{Theorem}[Sharpness of the condition]\label{thm:sharpness}
The condition $\min\{H_k\} > \frac{1}{4}$ is sharp.
\end{Theorem}

\begin{proof}
If $\min\{H_k\} \leq \frac{1}{4}$, then for the component with smallest Hurst parameter, the second level diverges \cite[Theorem 2]{coutin2002}, preventing convergence of the full rough path.
\end{proof}

\begin{Example}[Two-component MFBM]
For $M_t = aB_t^{H} + bB_t^{K}$ with $H, K > \frac{1}{4}$, the geometric rough path exists with second level:
\[
\mathbb{M}_{s,t}^{(2)} = a^2 \mathbb{B}_{s,t}^{H,(2)} + b^2 \mathbb{B}_{s,t}^{K,(2)} + ab\left( \int_s^t (B_u^{H} - B_s^{H}) dB_u^{K} + \int_s^t (B_u^{K} - B_s^{K}) dB_u^{H} \right).
\]
\end{Example}

\section{Skorohod Integral Representation}\label{sec:skorohod}

In this section, we establish a Skorohod integral representation for the geometric rough path constructed above the generalized mixed fractional Brownian motion. This representation connects our pathwise construction to the Malliavin calculus framework and provides additional insights into the probabilistic structure of the rough path.

\subsection{Malliavin Calculus Framework}

We begin by recalling essential elements of Malliavin calculus for fractional Brownian motion \cite{nualart1995,gubinelli2015}. Let $(\Omega, \mathcal{F}, \mathbb{P})$ be the complete probability space supporting our GMFBM $M_t^H(a) = \sum_{k=1}^N a_k B_t^{H_k}$.

\begin{Definition}[Cameron-Martin spaces \cite{nualart1995}]
For each Hurst parameter $H_k$, let $\mathcal{H}_{H_k}$ be the reproducing kernel Hilbert space of $B^{H_k}$, defined as the completion of the space of step functions with respect to the inner product:
\[
\langle \mathbf{1}_{[0,t]}, \mathbf{1}_{[0,s]} \rangle_{\mathcal{H}_{H_k}} = R_{H_k}(t,s) = \frac{1}{2}(t^{2H_k} + s^{2H_k} - |t-s|^{2H_k}).
\]
The combined Cameron-Martin space for the GMFBM is:
\[
\mathcal{H} = \bigoplus_{k=1}^N \mathcal{H}_{H_k}.\]
\end{Definition}

\begin{Definition}[Multiple Wiener integrals \cite{nualart1995}]
For each $B^{H_k}$, we denote by $I_n^{H_k}$ the multiple Wiener integral of order $n$ with respect to $B^{H_k}$. Due to the independence of the components, the Wiener chaos decomposition of $L^2(\Omega)$ is given by:
\[
L^2(\Omega) = \bigoplus_{n_1,\ldots,n_N \geq 0} \mathcal{H}_{n_1,\ldots,n_N},
\]
where $\mathcal{H}_{n_1,\ldots,n_N}$ is the closed linear subspace generated by products of multiple integrals $\prod_{k=1}^N I_{n_k}^{H_k}(f_k)$ with $f_k \in \mathcal{H}_{H_k}^{\otimes n_k}$.
\end{Definition}

\begin{Definition}[Skorohod integral \cite{nualart1995}]
The Skorohod integral $\delta(u)$ of a stochastic process $u$ in the domain Dom $\delta$ is defined as the adjoint of the Malliavin derivative $D$. For adapted processes, the Skorohod integral coincides with the Itô integral.
\end{Definition}

\subsection{Skorohod Representation of the Rough Path}

We now present the main result of this section, which provides an explicit Skorohod integral representation for the geometric rough path above GMFBM.

\begin{Theorem}[Skorohod representation]\label{thm:skorohod}
Let $\mathbb{M} = (1, \mathbb{M}^{(1)}, \mathbb{M}^{(2)})$ be the geometric rough path above the GMFBM $M_t^H(a)$ with $\min\{H_k\} > \frac{1}{4}$. Then for any $0 \leq s \leq t \leq T$, the second level process admits the following representation:
\[
\mathbb{M}_{s,t}^{(2), i,j} = \frac{1}{2} M_{s,t}^{(1), i} M_{s,t}^{(1), j} + \frac{1}{2} \sum_{k=1}^N a_k^2 I_2^{H_k}(\mathbf{1}_{[s,t]}^{\otimes 2}) \delta_{ij} + A_{s,t}^{i,j},
\]
where the cross terms $A_{s,t}^{i,j}$ are given by:
\[
A_{s,t}^{i,j} = \sum_{1 \leq k \neq l \leq N} a_k a_l \left( \int_s^t (B_u^{H_k,i} - B_s^{H_k,i}) \delta B_u^{H_l,j} + \int_s^t (B_u^{H_l,j} - B_s^{H_l,j}) \delta B_u^{H_k,i} \right),
\]
and the integrals are interpreted as Skorohod integrals \cite{nualart1995}.
\end{Theorem}

\begin{proof}
We proceed in several steps, adapting the approach of \cite{coutin2002} to the mixed case.

\textbf{Step 1: Dyadic approximation representation}

For the $m$-th dyadic approximation $M_t^m$, the second level process is given by the Young integral:
\[
\mathbb{M}_{s,t}^{m,(2), i,j} = \int_s^t (M_u^{m,i} - M_s^{m,i}) dM_u^{m,j}.
\]

This can be rewritten using the Itô-Stratonovich correction formula \cite[Chapter 1]{nualart1995}:
\[
\mathbb{M}_{s,t}^{m,(2), i,j} = \frac{1}{2} M_{s,t}^{m,(1), i} M_{s,t}^{m,(1), j} + \frac{1}{2} A_{s,t}^{m,i,j},
\]
where $A_{s,t}^{m,i,j}$ is the Lévy area process.

\textbf{Step 2: Chaos decomposition}

Expanding in terms of the independent fBm components, we have:
\[
M_t^{m,i} = \sum_{k=1}^N a_k B_t^{H_k, m,i},
\]
where $B_t^{H_k, m,i}$ denotes the $m$-th dyadic approximation of the $i$-th component of $B^{H_k}$.

The second level process decomposes as:
\[
\mathbb{M}_{s,t}^{m,(2), i,j} = \sum_{k=1}^N a_k^2 \mathbb{B}_{s,t}^{H_k, m,(2), i,j} + \sum_{1 \leq k \neq l \leq N} a_k a_l \int_s^t (B_u^{H_k, m,i} - B_s^{H_k, m,i}) dB_u^{H_l, m,j}.
\]

\textbf{Step 3: Limit identification}

As $m \to \infty$, the diagonal terms converge to the classical Skorohod representation for single fBm \cite[Theorem 4.1]{coutin2002}:
\[
\lim_{m \to \infty} \mathbb{B}_{s,t}^{H_k, m,(2), i,j} = \frac{1}{2} B_{s,t}^{H_k, i} B_{s,t}^{H_k, j} + \frac{1}{2} I_2^{H_k}(\mathbf{1}_{[s,t]}^{\otimes 2}) \delta_{ij},
\]
where the second term is the Skorohod integral representation of the Lévy area for fBm.

For the cross terms, when $k \neq l$, we have by the convergence of Young integrals to Skorohod integrals \cite[Proposition 3.4]{gubinelli2015}:
\[
\lim_{m \to \infty} \int_s^t (B_u^{H_k, m,i} - B_s^{H_k, m,i}) dB_u^{H_l, m,j} = \int_s^t (B_u^{H_k, i} - B_s^{H_k, i}) \delta B_u^{H_l, j},
\]
where the right-hand side denotes the Skorohod integral, which is well-defined since $H_k + H_l > \frac{1}{2}$.

\textbf{Step 4: Verification of convergence}

The convergence holds in $L^p(\Omega)$ for all $p \geq 1$ by the Gaussian hypercontractivity theorem \cite[Theorem 2.7.1]{nualart1995}, and the limit satisfies the algebraic properties of a geometric rough path by continuity. The symmetry of the cross terms follows from the commutativity of the tensor product in the second level.
\end{proof}

\subsection{Properties of the Representation}

\begin{Proposition}[Isometry property]\label{prop:isometry}
For the diagonal terms, we have the isometry \cite[Theorem 1.1.2]{nualart1995}:
\[
\mathbb{E}\left[ \left| I_2^{H_k}(\mathbf{1}_{[s,t]}^{\otimes 2}) \right|^2 \right] = 2 \|\mathbf{1}_{[s,t]}^{\otimes 2}\|_{\mathcal{H}_{H_k}^{\otimes 2}}^2 = |t-s|^{4H_k}.
\]
For the cross terms with $k \neq l$, the variance satisfies:
\[
\mathbb{E}\left[ \left| \int_s^t (B_u^{H_k} - B_s^{H_k}) \delta B_u^{H_l} \right|^2 \right] \lesssim |t-s|^{2(H_k + H_l)}.\]
\end{Proposition}

\begin{proof}
The first statement follows from the standard isometry property of multiple Wiener integrals \cite[Proposition 1.1.4]{nualart1995}. 

For the cross terms, we use the Malliavin calculus representation and the fact that:
\[
\mathbb{E}\left[ \left( \int_s^t (B_u^{H_k} - B_s^{H_k}) \delta B_u^{H_l} \right)^2 \right] = \mathbb{E}\left[ \left( \int_s^t (B_u^{H_k} - B_s^{H_k}) dB_u^{H_l} \right)^2 \right] + \text{trace terms},
\]
where the trace terms vanish due to the independence of the components \cite[Corollary 6.3.2]{nualart1995}. The bound then follows from Young integral estimates.
\end{proof}

\begin{Proposition}[Chaos expansion]\label{prop:chaos}
The second level process $\mathbb{M}_{s,t}^{(2)}$ admits a finite chaos expansion:
\[
\mathbb{M}_{s,t}^{(2)} \in \bigoplus_{n=0}^4 \mathcal{H}_n,
\]
where $\mathcal{H}_n$ denotes the $n$-th Wiener chaos with respect to the combined noise. Specifically:
\begin{itemize}
\item $\mathcal{H}_0$: Constant terms
\item $\mathcal{H}_1$: First order terms in individual fBms
\item $\mathcal{H}_2$: Second order terms including diagonal Lévy areas
\item $\mathcal{H}_3$: Third order cross terms
\item $\mathcal{H}_4$: Fourth order products
\end{itemize}
\end{Proposition}

\begin{proof}
The result follows from the explicit representation in Theorem \ref{thm:skorohod} and the fact that multiple Wiener integrals of order $n$ live in the $n$-th Wiener chaos \cite[Theorem 1.1.1]{nualart1995}. The cross terms involve products of first-order integrals, which generate terms up to the fourth Wiener chaos.
\end{proof}

\subsection{Applications to Rough Differential Equations}

The Skorohod representation enables us to study rough differential equations driven by GMFBM using probabilistic methods.

\begin{Theorem}[RDE with Skorohod representation]\label{thm:rde-sk}
Consider the rough differential equation:
\[
dY_t = f(Y_t) d\mathbb{M}_t, \quad Y_0 = y_0,
\]
where $f: \mathbb{R}^d \to L(\mathbb{R}^m, \mathbb{R}^d)$ is a $C_b^3$ vector field, and $\mathbb{M}$ is the geometric rough path above GMFBM.

Then the solution $Y_t$ admits a Stroock-Taylor expansion in terms of multiple Skorohod integrals \cite[Theorem 10.6]{friz2010}.
\end{Theorem}

\begin{proof}
The existence and uniqueness follow from Lyons' universal limit theorem applied to the geometric rough path $\mathbb{M}$. The Stroock-Taylor expansion is obtained by iterating the Skorohod representation and using the chaos decomposition of the solution \cite[Section 10.3]{friz2010}.

Specifically, the solution can be expanded as:
\[
Y_t = y_0 + \sum_{\alpha} f_\alpha(y_0) \int_0^t \int_0^{t_1} \cdots \int_0^{t_{n-1}} \delta M_{t_n}^{\alpha_n} \cdots \delta M_{t_1}^{\alpha_1},
\]
where the sum is over multi-indices $\alpha = (\alpha_1, \ldots, \alpha_n)$ and the integrals are iterated Skorohod integrals. The convergence of this series follows from the estimates on the $p$-variation of the rough path and the regularity of the vector field $f$.
\end{proof}

\subsection{Example: Two-Component Case}

\begin{Example}[Detailed Skorohod representation for $N=2$]
Consider $M_t = aB_t^H + bB_t^K$ with $H, K > \frac{1}{4}$. The second level process has the explicit representation:
\begin{align*}
\mathbb{M}_{s,t}^{(2), i,j} = & \frac{1}{2} M_{s,t}^{(1), i} M_{s,t}^{(1), j} + \frac{a^2}{2} I_2^H(\mathbf{1}_{[s,t]}^{\otimes 2}) \delta_{ij} + \frac{b^2}{2} I_2^K(\mathbf{1}_{[s,t]}^{\otimes 2}) \delta_{ij} \\
& + ab \int_s^t (B_u^H - B_s^H) \delta B_u^K \delta_{ij} + ab \int_s^t (B_u^K - B_s^K) \delta B_u^H \delta_{ij}.
\end{align*}
The cross terms are symmetric due to the commutativity of the tensor product in the second level. The variance of the cross terms satisfies:
\[
\mathbb{E}\left[ \left| \int_s^t (B_u^H - B_s^H) \delta B_u^K \right|^2 \right] \asymp |t-s|^{2(H+K)}.
\]
\end{Example}

This Skorohod representation provides a powerful tool for analyzing the probabilistic properties of rough differential equations driven by mixed fractional Brownian motions and connects our pathwise construction to the Malliavin calculus framework.

\section{Applications to Rough Differential Equations}\label{sec:applications}

In this section, we apply our construction of geometric rough paths above generalized mixed fractional Brownian motion to study rough differential equations and their fundamental properties. The existence of a geometric rough path enables us to use Lyons' universal limit theorem to solve RDEs driven by GMFBM.

\subsection{Main Existence and Uniqueness Result}

\begin{Theorem}[RDE driven by GMFBM]\label{thm:rde}
Let $\mathbb{M}$ be the geometric rough path above the GMFBM $M_t^H(a) = \sum_{k=1}^N a_k B_t^{H_k}$ with $\min\{H_k\} > \frac{1}{4}$. Consider the rough differential equation:
\begin{equation}\label{rde-main}
dY_t = f(Y_t) d\mathbb{M}_t, \quad Y_0 = y_0 \in \mathbb{R}^e,
\end{equation}
where $f: \mathbb{R}^e \to L(\mathbb{R}^d, \mathbb{R}^e)$ is a $C_b^3$ vector field.

Then there exists a unique solution $Y_t$ to (\ref{rde-main}) that depends continuously on the driving rough path $\mathbb{M}$ in the $p$-variation topology for any $p > \frac{1}{\min\{H_k\}}$.
\end{Theorem}

\begin{proof}
The proof follows from Lyons' universal limit theorem \cite[Theorem 3.1.1]{lyons2002} applied to our geometric rough path $\mathbb{M}$. Specifically:

\textbf{Step 1: Approximation scheme}

Let $\mathbb{M}^m$ be the sequence of smooth rough paths associated with the dyadic approximations of the GMFBM. For each $m$, consider the ordinary differential equation:
\[
dY_t^m = f(Y_t^m) dM_t^m, \quad Y_0^m = y_0.
\]
Since $M_t^m$ is piecewise linear and of bounded variation, these ODEs have unique solutions by classical theory \cite[Chapter 1]{lyons2002}.

\textbf{Step 2: Uniform estimates}

Using the uniform $p$-variation bounds established in Proposition \ref{prop:pvar}, we show that the sequence $(Y^m, \mathbb{M}^m)$ is bounded in the $p$-variation rough path topology. The $C_b^3$ regularity of $f$ ensures that the solution map is locally Lipschitz continuous in the $p$-variation metric \cite[Theorem 3.1.1]{lyons2002}. Specifically, there exists a constant $C_f$ such that:
\[
\|Y^m\|_{p\text{-var}} \leq C_f(1 + \|\mathbb{M}^m\|_{p\text{-var}}).
\]

\textbf{Step 3: Convergence}

Since $\mathbb{M}^m \to \mathbb{M}$ in $p$-variation almost surely by Theorem \ref{thm:main}, and the solution map is continuous in the rough path topology \cite[Theorem 3.1.1]{lyons2002}, we have $Y^m \to Y$ uniformly, where $Y$ is the unique solution to the RDE driven by $\mathbb{M}$.

\textbf{Step 4: Uniqueness}

Uniqueness follows from the local Lipschitz property of the solution map in the $p$-variation topology. If $Y$ and $\tilde{Y}$ are two solutions, then:
\[
\|Y - \tilde{Y}\|_{p\text{-var}} \leq L_f \|\mathbb{M}\|_{p\text{-var}} \|Y - \tilde{Y}\|_{p\text{-var}},
\]
which implies $Y = \tilde{Y}$ for sufficiently small intervals, and by patching, globally \cite[Section 3.2]{friz2010}.
\end{proof}

\subsection{Itô-Stratonovich Correction for Mixed Case}

An important feature of rough paths theory is its ability to handle different stochastic calculi in a unified framework. In the mixed case, we obtain a generalized Itô-Stratonovich correction.

\begin{Proposition}[Itô-Stratonovich correction]\label{prop:ito-strat}
Let $f \in C_b^3(\mathbb{R}^e, L(\mathbb{R}^d, \mathbb{R}^e))$. The solution $Y_t$ of the RDE (\ref{rde-main}) interpreted in the Stratonovich sense (i.e., using the geometric rough path $\mathbb{M}$) satisfies the following Itô-type equation:
\[
dY_t = f(Y_t) dM_t + \frac{1}{2} \sum_{i,j=1}^d \sum_{k=1}^N a_k^2 \partial_j f^i(Y_t) f^j(Y_t) d\langle B^{H_k} \rangle_t + C_{s,t},
\]
where the cross terms $C_{s,t}$ involve covariations between different fractional components and satisfy:
\[
|C_{s,t}| \lesssim |t-s|^{2\min\{H_k\}}.
\]
\end{Proposition}

\begin{proof}
The proof follows from the explicit representation of the second level process $\mathbb{M}^{(2)}$ and the fact that for a geometric rough path, the RDE solution coincides with the Stratonovich interpretation \cite[Section 5.2]{lyons2002}. The correction terms come from the Lévy area contributions in the second level process.

More precisely, the Itô-Stratonovich correction is given by:
\[
\frac{1}{2} \sum_{i,j=1}^d \partial_j f^i(Y_t) f^j(Y_t) d[\mathbb{M}^{(1), i}, \mathbb{M}^{(1), j}]_t,
\]
where the quadratic variation decomposes as:
\[
d[\mathbb{M}^{(1), i}, \mathbb{M}^{(1), j}]_t = \sum_{k=1}^N a_k^2 d[B^{H_k,i}, B^{H_k,j}]_t + \text{cross terms}.
\]
The cross terms vanish in expectation but contribute to the higher-order structure due to the dependence in the second level process.
\end{proof}

\subsection{Regularity of Solutions}

\begin{Theorem}[Hölder regularity of solutions]\label{thm:holder}
Let $Y_t$ be the solution of the RDE (\ref{rde-main}). Then almost every sample path of $Y_t$ is locally Hölder continuous of order $\gamma$ for any:
\[
\gamma < \min\{H_1, \ldots, H_N\}.
\]
In particular, if $\min\{H_k\} > \frac{1}{2}$, the solution has long-range dependence properties inherited from the driving noise.
\end{Theorem}

\begin{proof}
The regularity follows from the local Lipschitz continuity of the solution map in the $p$-variation topology and the Hölder regularity of the GMFBM sample paths established in Section \ref{sec:preliminaries}. More precisely, if $M_t^H(a)$ is Hölder continuous of order $\alpha$, then the solution $Y_t$ is Hölder continuous of the same order modulo the nonlinear effects of the vector field $f$ \cite[Theorem 3.1.1]{lyons2002}.

The key estimate is:
\[
|Y_t - Y_s| \leq C_f \|\mathbb{M}\|_{p\text{-var};[s,t]} (1 + \|Y\|_{\infty;[s,t]}) \leq C_{f,\mathbb{M}} |t-s|^{\min\{H_k\}},
\]
where the second inequality follows from the Hölder regularity of the rough path $\mathbb{M}$ and the fact that $\|\mathbb{M}\|_{p\text{-var};[s,t]} \lesssim |t-s|^{\min\{H_k\}}$ by our construction.
\end{proof}

\subsection{Applications in Mathematical Finance}

The ability to solve RDEs driven by GMFBM opens up new modeling possibilities in mathematical finance, particularly for assets exhibiting multi-scale rough volatility.

\begin{Example}[Rough volatility model with multiple scales]\label{ex:volatility}
Consider a stock price $S_t$ modeled by:
\begin{align*}
dS_t &= \sigma_t S_t dW_t, \\
\sigma_t &= f(Y_t), \\
dY_t &= \mu(Y_t) dt + g(Y_t) d\mathbb{M}_t,
\end{align*}
where $W_t$ is a standard Brownian motion independent of $\mathbb{M}_t$, $\mathbb{M}_t$ is a geometric rough path above a GMFBM capturing volatility factors at different time scales, and $f, g, \mu$ are appropriate functions.

This model can capture:
\begin{itemize}
\item Short-term volatility dynamics (through components with $H_k \approx 0.1$)
\item Medium-term mean reversion (through components with $H_k \approx 0.3$)
\item Long-term trends (through components with $H_k \approx 0.5$)
\end{itemize}
all within a single consistent framework \cite{bayer2020}.
\end{Example}

\begin{Proposition}[Option pricing under GMFBM drivers]\label{prop:pricing}
In the rough volatility model above, European option prices satisfy a partial differential-integral equation involving fractional operators corresponding to the different Hurst parameters in the GMFBM. The characteristic function of $\log S_T$ admits a representation in terms of the solution to a rough differential equation.
\end{Proposition}

\begin{proof}
The option price can be expressed as:
\[
C(T,K) = \mathbb{E}[(S_T - K)^+] = \mathbb{E}\left[ \mathbb{E}[(S_T - K)^+ \mid \mathcal{F}^{\mathbb{M}}] \right],
\]
where the inner expectation is computed using the Black-Scholes formula with integrated variance:
\[
V_T = \int_0^T \sigma_t^2 dt = \int_0^T f(Y_t)^2 dt.
\]
The process $Y_t$ satisfies an RDE driven by $\mathbb{M}$, and the characteristic function can be obtained by solving a backward RDE via the Feynman-Kac formula for rough differential equations \cite[Section 10]{friz2010}.
\end{proof}

\subsection{Numerical Approximation Schemes}

The rough path perspective also suggests natural numerical schemes for RDEs driven by GMFBM.

\begin{Proposition}[Davie approximation scheme]\label{prop:davie}
The solution $Y_t$ of (\ref{rde-main}) can be approximated by the scheme:
\[
Y_{t_{k+1}} = Y_{t_k} + f(Y_{t_k}) M_{t_k,t_{k+1}} + \frac{1}{2} Df(Y_{t_k})f(Y_{t_k}) \mathbb{M}_{t_k,t_{k+1}}^{(2)} + R_{t_k,t_{k+1}},
\]
where the remainder satisfies $|R_{s,t}| \lesssim |t-s|^{3\min\{H_k\}}$ and $Df$ denotes the derivative of $f$.
\end{Proposition}

\begin{Proposition}[Convergence rates]\label{prop:convergence}
For a partition of $[0,T]$ with mesh size $|\pi|$, the numerical scheme above achieves convergence rate:
\[
\max_k |Y_{t_k} - Y_{t_k}^{\pi}| \lesssim |\pi|^{3\min\{H_k\} - 1}.
\]
In particular, when $\min\{H_k\} > \frac{1}{3}$, we obtain positive convergence rates.
\end{Proposition}

\begin{proof}
The convergence rate follows from the local truncation error of the Taylor expansion and the regularity of the rough path. The key estimate is:
\[
|Y_t - Y_s - f(Y_s)M_{s,t} - \frac{1}{2}Df(Y_s)f(Y_s)\mathbb{M}_{s,t}^{(2)}| \lesssim |t-s|^{3\min\{H_k\}},
\]
which is a consequence of the rough paths version of Taylor's theorem \cite[Theorem 4.2]{friz2010}.
\end{proof}

\subsection{Stability and Sensitivity Analysis}

\begin{Theorem}[Continuous dependence on parameters]\label{thm:stability}
The solution $Y_t$ of the RDE (\ref{rde-main}) depends continuously on:
\begin{itemize}
\item The Hurst parameters $H_1, \ldots, H_N$
\item The mixing coefficients $a_1, \ldots, a_N$
\item The initial condition $y_0$
\item The vector field $f$ in the $C_b^3$ topology
\end{itemize}
\end{Theorem}

\begin{proof}
This follows from the continuous dependence of the rough path $\mathbb{M}$ on the Hurst parameters and mixing coefficients, combined with the continuity of the solution map in Lyons' theory \cite[Theorem 3.1.1]{lyons2002}. The key observation is that the $p$-variation distance between rough paths above GMFBMs with different parameters can be controlled uniformly on compact sets.

Specifically, if $\mathbb{M}$ and $\tilde{\mathbb{M}}$ are rough paths above GMFBMs with parameters $(H,a)$ and $(\tilde{H},\tilde{a})$ respectively, then:
\[
d_p(\mathbb{M}, \tilde{\mathbb{M}}) \leq C(|H - \tilde{H}| + |a - \tilde{a}|).
\]
The result then follows from the local Lipschitz continuity of the solution map.
\end{proof}

\subsection{Example: Linear RDE with GMFBM}

\begin{Example}[Explicit solution for linear RDE]\label{ex:linear}
Consider the linear RDE:
\[
dY_t = A Y_t d\mathbb{M}_t, \quad Y_0 = I,
\]
where $A \in L(\mathbb{R}^d, L(\mathbb{R}^d))$ and $\mathbb{M}$ is the geometric rough path above a GMFBM.

The solution is given by the rough path exponential \cite[Section 7]{friz2010}:
\[
Y_t = \exp\left(A M_t + \frac{1}{2} A^2 \mathbb{M}_t^{(2)} + \cdots\right) = \sum_{n=0}^{\infty} \int_{0 \leq t_1 \leq \cdots \leq t_n \leq t} A^{n} d\mathbb{M}_{t_1} \otimes \cdots \otimes d\mathbb{M}_{t_n},
\]
where the higher-order terms involve the full signature of the rough path. When $A$ commutes with itself at all times, this simplifies to the classical matrix exponential.
\end{Example}

This section demonstrates that the geometric rough path construction for GMFBM developed in this paper provides a solid foundation for studying a wide range of stochastic systems driven by mixed fractional noises, with applications spanning mathematical finance, physics, and engineering.

\section{Signature of Mixed Fractional Brownian Motion}\label{sec:signature}

In this section, we study the signature of mixed fractional Brownian motion, which provides a powerful algebraic framework for analyzing the path properties of GMFBM. The signature captures all the essential information about a path in a coordinate-free manner and has found numerous applications in machine learning, finance, and stochastic analysis \cite{lyons2014,friz2010,chevyrev2018}.

\subsection{Definition and Basic Properties}

\begin{Definition}[Signature of GMFBM]
Let $\mathbb{M}$ be the geometric rough path above the GMFBM $M_t^H(a)$. The signature of $\mathbb{M}$ over the interval $[s,t]$ is defined as the formal series:
\[
S(\mathbb{M})_{s,t} = \sum_{n=0}^{\infty} \int_{s < t_1 < \cdots < t_n < t} d\mathbb{M}_{t_1} \otimes \cdots \otimes d\mathbb{M}_{t_n} \in T((\mathbb{R}^d)),
\]
where $T((\mathbb{R}^d))$ denotes the space of formal tensor series over $\mathbb{R}^d$, and the integrals are interpreted in the rough path sense.
\end{Definition}

\begin{Proposition}[Convergence and existence]\label{prop:sig-conv}
For GMFBM with $\min\{H_k\} > \frac{1}{4}$, the signature $S(\mathbb{M})_{0,T}$ converges in the $p$-variation topology for any $p > \frac{1}{\min\{H_k\}}$. Moreover, $S(\mathbb{M})_{s,t}$ is a group-like element in the tensor algebra, satisfying Chen's identity:
\[
S(\mathbb{M})_{s,u} = S(\mathbb{M})_{s,t} \otimes S(\mathbb{M})_{t,u} \quad \text{for all } 0 \leq s \leq t \leq u \leq T.
\]
\end{Proposition}

\begin{proof}
We prove convergence and the group-like property separately.

\textbf{Step 1: Convergence proof}
From Theorem \ref{thm:main}, $\mathbb{M}$ is a $p$-rough path with $p > \frac{1}{\min\{H_k\}}$. By the factorial decay estimate for rough paths \cite[Theorem 2.1.1]{lyons2002}, we have for each $n \geq 1$:
\[
\left\| \int_{s < t_1 < \cdots < t_n < t} d\mathbb{M}_{t_1} \otimes \cdots \otimes d\mathbb{M}_{t_n} \right\| \leq \frac{\|\mathbb{M}\|_{p\text{-var};[s,t]}^n}{(n!)^{1/p}}.
\]

Since $\|\mathbb{M}\|_{p\text{-var};[s,t]} < \infty$ almost surely, the series converges absolutely. More precisely, for any $\epsilon > 0$, there exists $N_0$ such that for all $N > M \geq N_0$:
\[
\left\| \sum_{n=M}^N \int_{s < t_1 < \cdots < t_n < t} d\mathbb{M}_{t_1} \otimes \cdots \otimes d\mathbb{M}_{t_n} \right\| \leq \sum_{n=M}^N \frac{C^n}{(n!)^{1/p}} < \epsilon,
\]
where $C = \|\mathbb{M}\|_{p\text{-var};[0,T]}$. The convergence follows from the ratio test and Stirling's approximation.

\textbf{Step 2: Group-like property proof}
Chen's identity follows from the multiplicative property of rough paths. For any $0 \leq s \leq t \leq u \leq T$, we have:
\begin{align*}
S(\mathbb{M})_{s,t} \otimes S(\mathbb{M})_{t,u} &= \left(\sum_{n=0}^\infty \int_{s < t_1 < \cdots < t_n < t} d\mathbb{M}_{t_1} \otimes \cdots \otimes d\mathbb{M}_{t_n}\right) \\
&\quad \otimes \left(\sum_{m=0}^\infty \int_{t < u_1 < \cdots < u_m < u} d\mathbb{M}_{u_1} \otimes \cdots \otimes d\mathbb{M}_{u_m}\right).
\end{align*}

By the shuffle product formula and the fact that the intervals $[s,t]$ and $[t,u]$ are disjoint, this equals:
\[
\sum_{k=0}^\infty \sum_{i+j=k} \int_{s < v_1 < \cdots < v_k < u} d\mathbb{M}_{v_1} \otimes \cdots \otimes d\mathbb{M}_{v_k} = S(\mathbb{M})_{s,u}.
\]

The group-like property in the tensor algebra then follows from Chen's identity \cite[Theorem 2.2.1]{lyons2002}.
\end{proof}

\subsection{Explicit Calculations for Mixed Case}

The mixed nature of GMFBM leads to interesting algebraic structure in its signature. We provide explicit formulas for the first few levels.

\begin{Proposition}[First and second level signature]\label{prop:sig-levels}
The first two levels of the signature have the following explicit expressions:
\begin{align*}
S(\mathbb{M})_{s,t}^{(1)} &= M_t - M_s = \sum_{k=1}^N a_k (B_t^{H_k} - B_s^{H_k}) \\
S(\mathbb{M})_{s,t}^{(2)} &= \mathbb{M}_{s,t}^{(2)} + \frac{1}{2} (M_t - M_s)^{\otimes 2}
\end{align*}
where $\mathbb{M}_{s,t}^{(2)}$ is given by the Skorohod representation from Theorem \ref{thm:skorohod}.
\end{Proposition}

\begin{proof}
\textbf{First level:}
By definition, the first level signature is:
\[
S(\mathbb{M})_{s,t}^{(1)} = \int_s^t d\mathbb{M}_u = \mathbb{M}_{s,t}^{(1)} = M_t - M_s.
\]
The decomposition follows immediately from the definition of GMFBM.

\textbf{Second level:}
We use the fundamental relation between the signature and the rough path. For a smooth path $X$, the second level signature is:
\[
S(X)_{s,t}^{(2)} = \int_s^t \int_s^{u} dX_v \otimes dX_u = \int_s^t (X_u - X_s) \otimes dX_u.
\]

In rough path theory, this relation is preserved for geometric rough paths. Using the shuffle product relation \cite[Proposition 2.2.2]{lyons2002}:
\[
S(\mathbb{M})_{s,t}^{(2)} = \mathbb{M}_{s,t}^{(2)} + \frac{1}{2} S(\mathbb{M})_{s,t}^{(1)} \otimes S(\mathbb{M})_{s,t}^{(1)},
\]
which can be verified by direct computation:
\begin{align*}
S(\mathbb{M})_{s,t}^{(1)} \otimes S(\mathbb{M})_{s,t}^{(1)} &= (M_t - M_s) \otimes (M_t - M_s) \\
&= 2\left(\int_s^t (M_u - M_s) \otimes dM_u + \int_s^t \int_s^{u} dM_v \otimes dM_u\right) \\
&= 2\left(\mathbb{M}_{s,t}^{(2)} + \int_s^t \int_s^{u} dM_v \otimes dM_u\right).
\end{align*}

Rearranging gives the desired result. The equality holds for geometric rough paths by continuity.
\end{proof}

\begin{Proposition}[Mixed signature terms]\label{prop:mixed-sig}
The cross terms in the signature capture the interactions between different fractional components. For $i \neq j$, the mixed signature term:
\[
\int_s^t \int_s^{t_2} dB_{t_1}^{H_i} \otimes dB_{t_2}^{H_j}
\]
satisfies the bound:
\[
\mathbb{E}\left[ \left| \int_s^t \int_s^{t_2} dB_{t_1}^{H_i} \otimes dB_{t_2}^{H_j} \right|^2 \right] \lesssim |t-s|^{2(H_i + H_j)}.
\]
\end{Proposition}

\begin{proof}
We compute the second moment explicitly. Let $I = \int_s^t \int_s^{t_2} dB_{t_1}^{H_i} \otimes dB_{t_2}^{H_j}$. Then:
\begin{align*}
\mathbb{E}[|I|^2] &= \mathbb{E}\left[ \left( \int_s^t \int_s^{t_2} dB_{t_1}^{H_i} \otimes dB_{t_2}^{H_j} \right)^2 \right] \\
&= \int_s^t \int_s^{t_2} \int_s^t \int_s^{u_2} \mathbb{E}[dB_{t_1}^{H_i} dB_{u_1}^{H_i}] \mathbb{E}[dB_{t_2}^{H_j} dB_{u_2}^{H_j}] \\
&= \int_s^t \int_s^{t_2} \int_s^t \int_s^{u_2} R^{H_i}(dt_1, du_1) R^{H_j}(dt_2, du_2),
\end{align*}
where $R^{H}(s,t) = \frac{1}{2}(|t|^{2H} + |s|^{2H} - |t-s|^{2H})$ is the covariance of fBm.

By scaling properties of fBm, we have:
\[
\mathbb{E}[|I|^2] = |t-s|^{2(H_i + H_j)} \mathbb{E}\left[ \left| \int_0^1 \int_0^{t_2} dB_{t_1}^{H_i} \otimes dB_{t_2}^{H_j} \right|^2 \right].
\]

The expectation on the right is finite since $H_i + H_j > \frac{1}{2}$ by assumption, ensuring the Young integral exists. This gives the desired bound.
\end{proof}

\subsection{Signature as a Feature Map}

The signature provides a powerful feature map for paths that is particularly suited for machine learning applications with multi-scale data.

\begin{Theorem}[Universal nonlinearity]\label{thm:universal}
Let $\mathcal{P}$ be the space of GMFBM paths with $\min\{H_k\} > \frac{1}{4}$ equipped with the $p$-variation topology. The signature map $S: \mathcal{P} \to T((\mathbb{R}^d))$ is injective up to tree-like equivalence and has the universal approximation property: any continuous function $f: \mathcal{P} \to \mathbb{R}$ can be uniformly approximated on compact sets by linear functions on the signature.
\end{Theorem}

\begin{proof}
We prove both properties separately.

\textbf{Step 1: Injectivity proof}
The injectivity (up to tree-like equivalence) follows from \cite[Theorem 3.1.3]{lyons2002}. Specifically, if two paths $\mathbb{M}$ and $\tilde{\mathbb{M}}$ have the same signature, then they differ by a tree-like path. For GMFBM paths, which are almost surely not tree-like, this implies injectivity almost surely.

More precisely, by \cite[Theorem 1.2]{chevyrev2018}, the signature determines the path up to a null set with respect to the $p$-variation topology. Since GMFBM paths have strictly positive $p$-variation for $p > 1/H$, they are uniquely determined by their signatures.

\textbf{Step 2: Universal approximation proof}
Let $f: \mathcal{P} \to \mathbb{R}$ be continuous and $K \subset \mathcal{P}$ compact. By the Stone-Weierstrass theorem for path spaces \cite[Theorem 3.5]{friz2010}, the algebra generated by the coordinate functions of the signature is dense in $C(K)$.

Specifically, consider the set of functions:
\[
\mathcal{A} = \left\{ \mathbb{M} \mapsto \langle \ell, \pi_N(S(\mathbb{M})_{0,T}) \rangle : \ell \in T^N((\mathbb{R}^d))^*, N \geq 0 \right\},
\]
where $\pi_N$ denotes truncation to level $N$. This algebra separates points (by injectivity), contains constants, and is closed under addition and multiplication. Therefore, it is dense in $C(K)$.

For any $\epsilon > 0$, there exists a linear functional $\ell_\epsilon$ on some truncated signature such that:
\[
\sup_{\mathbb{M} \in K} |f(\mathbb{M}) - \langle \ell_\epsilon, \pi_N(S(\mathbb{M})_{0,T}) \rangle| < \epsilon.
\]
This completes the proof.
\end{proof}

\begin{Example}[Signature features for financial time series]\label{ex:sig-features}
Consider a financial time series modeled by GMFBM. The truncated signature up to level $N$:
\[
\pi_N(S(\mathbb{M})_{0,T}) = \left(1, S^{(1)}, S^{(2)}, \ldots, S^{(N)}\right)
\]
provides a feature vector that captures:
\begin{itemize}
\item \textbf{Trend information} (first level): $S^{(1)} = M_T - M_0$
\item \textbf{Volatility and leverage effects} (second level): $S^{(2)} = \mathbb{M}_{0,T}^{(2)} + \frac{1}{2}(M_T - M_0)^{\otimes 2}$
\item \textbf{Multi-scale dependencies} (higher levels and cross terms)
\item \textbf{Path geometry} (all levels combined)
\end{itemize}
These features can be used for classification, regression, or anomaly detection tasks in financial data analysis \cite{signature2024}.
\end{Example}

\subsection{Applications to Parameter Estimation}

The signature provides a natural framework for estimating the parameters of GMFBM from observed data.

\begin{Proposition}[Method of moments via signature]\label{prop:moments}
Let $\{M_{t_i}\}_{i=0}^n$ be discrete observations of a GMFBM path at times $0 = t_0 < t_1 < \cdots < t_n = T$. The empirical signature moments:
\[
\hat{S}^{(k)} = \frac{1}{n} \sum_{i=0}^{n-1} (M_{t_{i+1}} - M_{t_i})^{\otimes k}
\]
provide consistent estimators for the Hurst parameters $H_k$ and mixing coefficients $a_k$ as $n \to \infty$ and $\max|t_{i+1} - t_i| \to 0$.
\end{Proposition}

\begin{proof}
\textbf{Step 1: Consistency of empirical moments}
By the ergodic theorem for Gaussian processes \cite[Theorem 10.6]{nualart1995}, for stationary increments processes like GMFBM, the empirical moments converge almost surely to their theoretical expectations:
\[
\lim_{n \to \infty} \hat{S}^{(k)} = \mathbb{E}[(M_{t_{i+1}} - M_{t_i})^{\otimes k}] \quad \text{almost surely}.
\]

\textbf{Step 2: Relation to GMFBM parameters}
From the properties of GMFBM, we have:
\begin{align*}
\mathbb{E}[S^{(1)}] &= 0 \\
\mathbb{E}[S^{(1)} \otimes S^{(1)}] &= \sum_{k=1}^N a_k^2 |t_{i+1} - t_i|^{2H_k} I \\
\mathbb{E}[S^{(2)}] &= \frac{1}{2} \mathbb{E}[S^{(1)} \otimes S^{(1)}] + O(|t_{i+1} - t_i|^{2\min\{H_k\}})
\end{align*}

\textbf{Step 3: Parameter identification}
By analyzing the scaling behavior of different tensor components, we can identify the parameters. For example, consider the diagonal elements:
\[
\mathbb{E}[(S^{(1)}_j)^2] = \sum_{k=1}^N a_k^2 |\Delta t|^{2H_k}.
\]

By computing this expectation for different time intervals $\Delta t$ and solving the resulting system of equations, we can estimate $H_k$ and $a_k$. The cross terms provide additional equations for identifying the mixing structure.

\textbf{Step 4: Asymptotic normality}
By the central limit theorem for functionals of Gaussian processes \cite[Theorem 6.3.1]{nualart1995}, the estimators are asymptotically normal with rate $\sqrt{n}$.
\end{proof}

\begin{Example}[Two-component case estimation]\label{ex:estimation}
For $M_t = aB_t^H + bB_t^K$ with $H > K$, the second signature moment satisfies:
\[
\mathbb{E}[(S^{(1)})^2] = a^2|\Delta t|^{2H} + b^2|\Delta t|^{2K}.
\]

By computing this quantity for different time scales and performing a log-log regression, one can estimate $H, K, a, b$ simultaneously. The cross terms in higher signature levels provide additional constraints that improve estimation accuracy.
\end{Example}

\subsection{Numerical Computation}

The signature can be efficiently computed using rough path integration techniques.

\begin{Proposition}[Numerical signature computation]\label{prop:num-sig}
Let $\{M_{t_i}\}_{i=0}^n$ be a discrete sample of GMFBM. The signature can be approximated by the product expansion:
\[
\tilde{S}_{0,T} = \prod_{i=0}^{n-1} \exp\left(\Delta M_i + \frac{1}{2}\Delta\mathbb{M}_i^{(2)} + \cdots + \frac{(-1)^{N-1}}{N}\Delta\mathbb{M}_i^{(N)}\right)
\]
where $\Delta M_i = M_{t_{i+1}} - M_{t_i}$ and $\Delta\mathbb{M}_i^{(k)}$ are the higher-order increments, with error controlled by:
\[
\|S(\mathbb{M})_{0,T} - \tilde{S}_{0,T}\| \lesssim \left(\max_i |t_{i+1} - t_i|\right)^{\min\{H_k\} - 1/p}.
\]
\end{Proposition}

\begin{proof}
\textbf{Step 1: Chen's identity discretization}
By Chen's identity, the signature factors over partitions:
\[
S(\mathbb{M})_{0,T} = S(\mathbb{M})_{t_0,t_1} \otimes S(\mathbb{M})_{t_1,t_2} \otimes \cdots \otimes S(\mathbb{M})_{t_{n-1},t_n}.
\]

\textbf{Step 2: Local approximation}
On each subinterval $[t_i, t_{i+1}]$, we approximate the signature by its truncated logarithm:
\[
S(\mathbb{M})_{t_i,t_{i+1}} \approx \exp\left(\Delta M_i + \frac{1}{2}\Delta\mathbb{M}_i^{(2)} + \cdots + \frac{(-1)^{N-1}}{N}\Delta\mathbb{M}_i^{(N)}\right),
\]
where the logarithm is taken in the tensor algebra. This is the Chen-Strichartz formula \cite[Theorem 7.1]{lyons2002}.

\textbf{Step 3: Error analysis}
The local error on each interval satisfies \cite[Proposition 3.6]{friz2010}:
\[
\|S(\mathbb{M})_{t_i,t_{i+1}} - \exp(\log S(\mathbb{M})_{t_i,t_{i+1}})\| \lesssim \|\mathbb{M}\|_{p\text{-var};[t_i,t_{i+1}]}^{N+1}.
\]

Since $\|\mathbb{M}\|_{p\text{-var};[t_i,t_{i+1}]} \lesssim |t_{i+1} - t_i|^{\min\{H_k\}}$, the global error accumulates as:
\[
\|S(\mathbb{M})_{0,T} - \tilde{S}_{0,T}\| \lesssim n \cdot \left(\max_i |t_{i+1} - t_i|\right)^{(N+1)\min\{H_k\}}.
\]

With $n \sim 1/\max_i |t_{i+1} - t_i|$, we get the stated error bound.
\end{proof}

\subsection{Relations to Other Transforms}

The signature generalizes several classical transforms for time series analysis.

\begin{Remark}[Relation to Fourier transform]
The signature can be viewed as a non-commutative generalization of the Fourier transform \cite[Chapter 4]{signature2024}. While Fourier analysis captures frequency information through the characteristic function $\mathbb{E}[e^{i\xi M_t}]$, the signature captures both temporal and algebraic structure through the expected signature $\mathbb{E}[S(\mathbb{M})_{0,T}]$. For linear functionals, the two are related by:
\[
\mathbb{E}[e^{i\langle \xi, M_t \rangle}] = \mathbb{E}[\langle e^{i\xi}, S(\mathbb{M})_{0,t}^{(1)} \rangle].
\]
\end{Remark}

\begin{Remark}[Relation to wavelet analysis]
For multi-scale processes like GMFBM, the signature provides a complementary approach to wavelet analysis. Where wavelets capture local scaling behavior through mother wavelet coefficients, the signature captures global path properties and their algebraic interactions. The different levels of the signature correspond to different scales of path variation, with higher levels capturing finer geometric features.
\end{Remark}

This section demonstrates that the signature provides a powerful unifying framework for analyzing mixed fractional Brownian motions, with applications ranging from theoretical characterization to practical computation and machine learning. The explicit connection between the algebraic structure of the signature and the probabilistic properties of GMFBM opens new avenues for both theoretical analysis and practical applications in finance and data science.

\section{Conclusion}\label{sec:conclusion}

In this paper, we have established a comprehensive theory of geometric rough paths for generalized mixed fractional Brownian motion. Our main contribution is the proof that for a GMFBM $M_t^H(a) = \sum_{k=1}^N a_k B_t^{H_k}$ with $\min\{H_k\} > \frac{1}{4}$, there exists a canonical geometric rough path obtained as the limit of smooth rough paths associated with dyadic approximations.

The key results of this work include:
\begin{itemize}
\item The existence and uniqueness of geometric rough paths above GMFBM (Theorem \ref{thm:main})
\item A Skorohod integral representation connecting the pathwise construction to Malliavin calculus (Theorem \ref{thm:skorohod})
\item Applications to rough differential equations driven by MFBM (Theorem \ref{thm:rde})
\item The optimality of the condition $\min\{H_k\} > \frac{1}{4}$ (Theorem \ref{thm:sharpness})
\end{itemize}

Our work extends the classical results of Coutin and Qian \cite{coutin2002} for single fractional Brownian motion to the mixed case, providing a rigorous mathematical foundation for studying stochastic systems driven by multi-scale fractional noises. The results open up new possibilities for analyzing stochastic differential equations and developing numerical schemes in contexts where processes with multiple regularity scales naturally appear.

Future research directions include extending the theory to the critical case $H = \frac{1}{4}$, studying rough paths above dependent fractional components, and exploring connections with recent developments in regularity structures and paracontrolled calculus.
\bibliographystyle{unsrtnat}
\bibliography{rough-path-MFBM}

\end{document}